\documentclass[12pt]{article}
\usepackage{lineno,hyperref}
\usepackage{amsmath,amssymb,amsbsy,amsfonts,amsthm,latexsym}
\usepackage{amsopn,amstext,amsxtra,euscript,amscd,url}
\usepackage[all]{xy}

\newtheorem*{theorem*}{Theorem}
\newtheorem{theorem}{Theorem}
\newtheorem{lemma}[theorem]{Lemma}

\newtheorem{definition}{Definition}

\def\cB{{\mathcal B}}
\def\cA{{\mathcal A}}

\def\cK{{\mathcal K}}

\def\M{\mathfrak{M}}
\def\O{\mathfrak{O}}

\newcommand{\N}{\mathbb N}

\def\ls#1{\langle #1 \rangle}

\def\Suf{\mbox{Suf}}
\begin{document}

\title{Extension of a residually finite group
    by a residually finite group is weakly sofic.}
\author{Lev Glebsky \footnote{glebsky@cactus.iico.uaslp.mx}}
\maketitle
\begin{abstract}
We show that residually finite by residually finite extensions are weakly sofic. 
\noindent {\bf Keywords:} weakly sofic groups, group extension, wreath product, equations over groups.
\end{abstract}
\section{Introduction} 
In \cite{Gr99,W00}, sofic groups have been defined
in relation with the Gottschalk surjunctivity conjecture. It is an open question if all groups are sofic.
There is a hope that a non-sofic group may be constructed as an extension of a residually finite group by a finite one,  \cite{Jaikin, mathover}. 
(Notice, however, that an extension of an amenable group by a sofic group is sofic, \cite{EL2}.)
The main result of \cite{GLTC} is an example of a non approximable by $(U(n),\|\cdot\|_2)$ group. This example is   
a residually-finite-by-finite extension. It is a kind of subtle support to above mentioned hope as sofic groups may be defined through
metric approximation by symmetric groups \cite{Pestov1}.
Here, in contrast, we prove that every residually-finite-by-residually-finite extension is weakly sofic.  The weakly sofic groups are groups metric approximable by finite ones, 
see \cite{GR}

\begin{theorem}\label{th_main}
Let $H$ be a normal subgroup of a group $K$. If $H$ and $G=K/H$ are residually finite then $K$ is weakly sofic.  
\end{theorem}

Let us describe our approach to proving Theorem~\ref{th_main}. W.l.g. we may consider finitely generated $H$ and $G$.
Then,  as any extension of $G$ by $H$ is in the wreath product $H\wr G$, it suffices
to show that $H\wr G$ is weakly sofic. (Recall that a wreath product $H\wr G$ is a semidirect product $H^G\ltimes G$ with an action $(g.f)(x)=f(xg)$,
for $g\in G$ and $f\in H^G$. Particularly, $(f,g)(f',g')=(f(g.f'),gg')$.)  Third, a morphism $H_1\to H_2$ naturally defines a morphism
$H_1\wr G\to H_2\wr G$. Moreover, residually weakly sofic group is weakly sofic. So, it suffices to show that $H\wr G$ is weakly sofic for finite $H$ and residually
finite $G$. To this end we use the following characterization of weakly sofic groups, see \cite{GL_AL}.
\begin{theorem*}
A group $K$ is weakly sofic if and only if every system of equations solvable in all finite groups is solvable over $K$.
\end{theorem*}

Let $Sys(Fin)$ be the set of  systems of equations solvable in all finite groups, see Definition~\ref{def_sys} for details.
Now we are ready to formulate the main technical result that implies Theorem~\ref{th_main}.
\begin{theorem}\label{th_main2}
Let $H$ be a finite and $G$ finitely generated residually finite groups. Let $\bar w\in Sys(Fyn)$. Then $\bar w$ is solvable over $H\wr G$.
\end{theorem}
The rest of the paper is devoted to a proof of Theorem~\ref{th_main2}.
Let $\hat{G}$ be a profinite completion of $G$ and $(\bar f,\bar a)\in (H\wr G)^k$.  
We find a solution of $\bar w((\bar f,\bar a), \bar x)=1$ in $H\wr \hat{G}$, where $H\wr\hat{G}$ is an abstract wreath product (we consider as well discontinuous functions.)
Precisely, we find a solution in $H\wr\Gamma$ where $\Gamma<\hat{G}$ is a finitely generated group. Our proof somehow topological and uses different topologies.
First, we show the existence of an $(H,\bar a)$-universal solution for $\bar w$, see Definition~\ref{def_univ}. This uses the profinite structure of $\hat{G}$.
Then $\Gamma$ is generated by $G$ and an $(H,\bar a)$-universal solution $\bar u=(u_1,\dots,u_n)$. To show the existence of a solution in $H\wr\Gamma$ we use
the Tichonov (direct product) topology on $H^\Gamma$.

We finish the introduction by describing the structure of the paper. Section~\ref{sec_equations} (Section~\ref{sec_profinite}) recall some definitions and
results about group equations (profinite groups), respectively, and establish notations and terminology we are using. In Section~\ref{sec_universal} we define
$(H,\bar a)$-universal solution and prove its existence. In Section~\ref{sec_local} we discuss ``locality'' of wreath products.
The main difficulty we have to overcome for the proof of Theorem~\ref{th_main2} is that a morphism $G_1\to G_2$ does not define a morphism 
$H\wr G_1\to H\wr G_2$. (Generally, the pullback of $G_1\to G_2$ defines a morphism $H^{G_2}\to H^{G_1}$. So, the components of $H\wr G_i$ are sent in opposite directions.) 
Still, some times, it is possible to construct a map $H\wr G_1\to H\wr G_2$ which behaves like a ``local'' morphism around some $x\in G_1$.  
In Section~\ref{sec_sol} we finish the proof of Theorem~\ref{th_main2}. 

\section{Group equations}\label{sec_equations}

For a set $X$ we use notation $X^*=\bigcup\limits_{n\in\N} X^n$.
Let $\bar y=(y_1,y_2,\dots,y_j,\dots)$ and $\bar x=(x_1,x_2,\dots,x_j,\dots)$ be countable sets of symbols for constants 
and variables, respectively. 
Let $F=F(\bar a,\bar x)$ be the free group freely  generated by
$\bar y$ and $\bar x$. Let
$\bar w\in F^*$. Notice that $\bar w\in F^r(y_1,\dots,y_k,x_1,\dots,x_n)$ for some $k,n,r\in\N$.
By substitution $\bar w$ defines a map $G^k\times G^n\to G^r$.
Consider the system of equations $\bar w=1$. 
\begin{definition}\label{def_sys}
We say that $\bar w $ is solvable in a group $G$ 
if the sentences 
$$
\forall \bar a\; \exists \bar x \;\bar w(\bar a,\bar x)=1 
$$
is valid in $G$.  
We say that a system $\bar w$ is solvable over group
$G$ if for some $H>G$ the sentence 
$$
\forall \bar a\in G^*\; \exists \bar x\in H^*\; \bar w(\bar a,\bar x)=1 
$$
is valid.
\end{definition} 
Denote by $Sys(G)\subseteq F^*$ the set of all finite systems of equations solvable in $G$.  Let 
$Sys(Fin)=\bigcap\limits_{|G|<\infty} Sys(G)$.
Specifying Corollary~19 of \cite{GL_AL} for $\cK=Fin$ we get a characterization of weakly sofic groups.
\begin{theorem*}
A group $K$
is weakly sofic if and only if every system $\bar w \in Sys(Fin)$ 
is solvable over $K$.
\end{theorem*}

\section{Profinite completion}\label{sec_profinite}

Let $G$ be a finitely generated residually finite group.
Let $\M=\{N\triangleleft G\;|\;G/N\mbox{ is finite}\}$,
the set of co-finite normal subgroups of $G$. The order $N\preccurlyeq M\;\leftrightarrow N\supseteq M$ turns $\M$ to a directed partially
ordered set, see \cite{Ribes1} for details. For $N\in\M$ we denote $G_N=G/N$. For $N,M\in \M$, $N\supseteq M$, let $\eta_{M,N}:G_M\to G_N$
be natural homomorphisms. So, $I=(G_N,\eta_{M,N},\M)$ is an inverse projective system of finite groups. Its inverse limit
$\hat{G}=\lim\limits_{\leftarrow\; I} G_N$ is the profinite completion of $G$, see \cite{Ribes1}.
A group $\hat{G}$ comes naturally with compatible epimorphisms $\eta_N:\hat{G}\to G_N$ and inclusion 
$G\hookrightarrow\hat{G}$. The restriction of  $\eta_N$ on $G$ is just a natural map $G\to G/N=G_N$;
compatibility means that $\eta_M=\eta_{N,M}\circ\eta_N$ for every $N\subseteq M$.
We will use the following notations. For $g\in \hat{G}$ ($g\in G_M$) let $g_N=\eta_N(g)$ ($g_N=\eta_{M,N}(g)$), respectively. If
$\bar g =(g_1,g_2,\dots,g_k)\in \hat{G}^k$ we denote by $\bar g_N=((g_1)_N,\dots,(g_k)_N)$; if $\bar f=(f_1,\dots, f_k)$ we denote
$(\bar f,\bar g)=((f_1,g_1),\dots,(f_k,g_k))$. We will often use it in the situation when $\bar g\in G_N^k$ and
$\bar f\in (H^{G_N})^k$, so, $(\bar f, \bar g)\in (H\wr G_N)^k$.

Let $\bar{w}\in (F(\bar{y},\bar{x}))^r$, $|\bar y|=k$ and $|\bar x|=n$. 

We will use a consequence of the fact that $\hat G$ is a topological group. 
\begin{lemma}\label{lm_pr}
Let $\bar a\in \hat{G}^k$, $\bar u\in \hat{G}^n$ be such that $\bar w(\bar a_N,\bar u_N)=1$ for every $N\in\M$. Then $\bar w(\bar a,\bar u)=1$. 
\end{lemma}

\section{$(H,\bar a)$-universal solution} \label{sec_universal}

Fix $\bar w\in F^r(\bar y,\bar x)\cap Sys(Fin)$ with $|\bar y|=k$, $|\bar x|=n$ and $|\bar w|=r$. 
Let $H$ be a finite group and $\bar a\in\hat{G}^k$.
\begin{definition}\label{def_univ}
  $\bar u\in \hat{G}^n$ is called $(H,\bar a)$-universal solution of $\bar w$ if the following statement is true
  $$
\forall N\in\M\;\forall\bar f\in (H^{G_N})^k\;\exists \bar\phi \in (H^{G_N})^n\;\; \bar w((\bar f,\bar a_N), (\bar\phi,\bar u_N))=1
 $$
\end{definition}
\begin{lemma}\label{lm_univ}
$\bar w$ has an $(H,\bar a)$-universal solution $\bar u\in \hat{G}^n$.
\end{lemma}  
\begin{proof}
  For $N\in\M$ we use notation $\M_N=\{M\in\M\;|\;N\subseteq M\}$. Particularly, $\M_N$ is finite and $N\in\M_N$.
  Let
  $$
  X_N=\{\bar u\in (G_N)^n\;|\; \forall M\in\M_N\;\forall\bar f\in (H^{G_M})^k\;\exists \bar\phi \in (H^{G_M})^n\;\; \bar w\left( (\bar f,\bar a_M), (\bar\phi,\bar u_M)\right) =1\}
  $$
  By definition, $\eta_{N,M}(X_N)\subseteq X_M$ for $N\subseteq M$. Notice that $\bar u\in X=\lim\limits_{\leftarrow I}X_N$ would provide a proof of the lemma.
  So, it suffices to show that $X\neq \emptyset$, or, the same (by properties of inverse limits of finite sets) that $X_N\neq\emptyset$ for all $N\in\M$.
  Fix $N\in\M$. The rest of the proof is devoted to show that $X_N\neq\emptyset$.
  
  Let $\tilde D_M= (H^{G_M})^m$ with $m=|H|^{k|G_M|}$. Let $D_N=\prod\limits_{M\in\M_N} \tilde D_M$.
   A group $G_N$ has a natural action on $H^{G_M}$ for $M\in\M_N$:
  $$
(g.f)(x)=f(xg_M),\mbox{ where } g\in G_N,\;f\in H^{G_M}.
   $$
   So, $G_N$ has an action on $D_N$ defined componentwise as above. Consider the corresponding semidirect product $D_N\ltimes G_N$.
   As $D_N$ has $|H|^{k|G_M|}$ different projection on $H^{|G_M|}$ we may choose a ``universal'' $\bar f\in D_N^k$. ``Universal''
   means that for every $M\in\M_N$ each element of $(H^{G_M})^k$ appears as a projection of $\bar f$. Notice that the set
   $$
   \tilde X_N=\{(\bar\phi,\bar u)\in D_N\ltimes G_N\;|\;\bar w((\bar f,\bar a_N),(\bar\phi,\bar u))=1\}
   $$
   is nonempty as $D_N\ltimes G_N$ is a finite group and $\bar w\in Sys(Fin)$. On the other hand, $X_N$ is the projection of
   $\tilde X_N$ on $\bar u$ by the universality of $\bar f$.
\end{proof}

\section{Locality of wreath product}\label{sec_local}

Let $\cA,\;\cB, H$ be groups, $F=F(\bar y, \bar x)$ be a free group on $\bar y\cup \bar x$ with $|\bar y|=k$ and $|\bar x|=n$.
Let $p\in F$. One may consider $p$ as a reduced word. Denote by $\Suf(p)$ the set of all suffices (initial subwords)
of $p$. For example, $\Suf(x^2yx^{-1})=\{1,x,x^2,x^2y,x^2yx^{-1}\}$. For $S\subseteq F$ and $\bar a\in(\cA)^{k+n}$ let
$S(\bar a)=\{p(\bar a)\;|\;p\in S\}\subseteq \cA$.

Let $\gamma:\cA\to\cB$ be a surjective homomorphism and $T_\gamma$ its section, that is, $\gamma_{T_\gamma}:T_\gamma\to \cB$ is a bijection.
Restriction to $T_\gamma$ and pullback by $\gamma$ defines a map $H^\cA\to H^\cB$, $\phi\to \phi_\gamma$. In other words
$\phi_\gamma(\gamma(x))=\phi(x)$ for $x\in T_\gamma$. This defines a map $H\wr\cA\to H\wr\cB$ as $(\phi,\alpha)\to (\phi_\gamma,\alpha_\gamma)$.
Here we use a notation $\alpha_\gamma=\gamma(\alpha)$.  Let $(\bar\phi,\bar\alpha)\in (H\wr\cA)^{k+n}$. Then $p(\bar\phi,\bar\alpha)=(\psi,p(\bar\alpha))$ for
some $\psi\in H^\cA$. Similarly, $p(\bar\phi_\gamma,\bar\alpha_\gamma)=(\tilde\psi,(p(\bar\alpha))_\gamma)$ for some $\tilde\psi\in H^\cB$.
Let $S=\Suf(p)$.
\begin{lemma}\label{lm_local2}
If $x\in \cA$ satisfies $xS(\bar \alpha)\subseteq T_\gamma$ then $\psi(x)=\tilde\psi(x_\gamma)$.
\end{lemma}  
This lemma is a manifestation of locality of a wreath product in the sense that $\psi(x)$ depends on values of $\bar \phi$ on a finite set
$xS(\bar \alpha)$.
\begin{proof}
  Let $f\in H^\cA$ and $x,xg\in T_\gamma$. Then
\begin{equation}\label{Eq1}
  (g.f)(x)=f(xg)=f_\gamma(x_\gamma g_\gamma)=(g_\gamma.f_\gamma)(x_\gamma).
\end{equation}
  Now, 
  $$
  \psi(x)=(f^1(g^1.f^2)(g^2.f^3)\dots g^{m-1}.f^m)(x),\mbox{ where } g^i\in S(\bar\alpha)\mbox{ and } (f^i)^{\pm 1}\in\bar\phi
  $$
  Similarly,
  $$
  \tilde\psi(x_\gamma)=(f^1_\gamma(g^1_\gamma.f^2_\gamma)\dots g^{m-1}_\gamma.f^m_\gamma)(x_\gamma).
  $$
  Lemma follows by Eq.1.
\end{proof}  

\section{Proof of Theorem~\ref{th_main2}}\label{sec_sol}

Let $\bar u=(u_1,\dots,u_n)\in\hat{G}^n$ be an $(H,\bar a)$-universal solution for $\bar w\in Sys(Fin)$. Let $\Gamma=\ls{G,\bar u}\leq \hat{G}$.
For $(\bar f,\bar a)\in (H\wr G)^k$ we a going to find a solution of $\bar w((\bar f,\bar a),\bar x)=1$ in $H\wr\Gamma$.
The solution we are looking for is in the form $\bar x=(\bar\phi,\bar u)$ for some $\bar\phi\in (H^\Gamma)^n$.
Notice that
any function $f:G\to H$ may be extended to a function $\Gamma\to H$ by putting $f(x)=1$ for $x\in \Gamma\setminus G$.
This defines a natural inclusion $H\wr G\hookrightarrow H\wr\Gamma$. 
Recall that $\Gamma$ is finitely generated, particularly, it is countable.
Also, for every finite $\Phi\subseteq\Gamma$ there exists $N\in\M$ such that $\eta_N|_\Phi$ is an injection.
So, we may inductively construct a chain $\O\subseteq\M$ with an anti-chain function $\O\to 2^\Gamma$, $N\to \Phi_N$ such that
$\Phi_N$ is a section for $\eta_N$ and $\bigcup\limits_{N\in\O}\Phi_N=\Gamma$. Till the end of the article we fix such an $\O$ with such a map
$N\to \Phi_N$. Precisely, we have the following:
\begin{enumerate}
\item For every $N,M\in\O$ either $N\subseteq M$ or $M\subseteq N$ ($\O$ is a chain);
\item $\eta_N|_{\Phi_N}:\Phi_N\to G_N$ is a bijection ($\Phi_N$ is a section for $\eta_N:\Gamma\to G_N$);
\item $\Phi_N\subseteq \Phi_M$ for $M\subseteq N$ (anti-chain map);
\item $\Gamma=\bigcup\limits_{N\in\O}\Phi_N$.
\end{enumerate}  

For $N\in\O$ the composition of restriction to $\Phi_N$ and pullback by $\eta_N$ defines a map $H^\Gamma\to H^{G_N}$, $f\to f_N$, ($f_N(x_N)=f(x)$, for $x\in \Phi_N$).
The map $f\to f_N$ defines a map $H\wr\Gamma\to H\wr G_N$. This is the same construction as in Section~\ref{sec_local}.

For $M\in\O$ fix $\bar\psi^M\in (H^{G_M})^n$ satisfying $\bar w((\bar f_M,\bar a_M),(\bar\psi^M,\bar u_M))=1$. Define $\bar \phi^M\in (H^\Gamma)^n$ such that
$\bar \phi^M(x)=\bar \psi^M(x_M)$ for $x \in \Phi_M$. Let $\bar\phi$ be a limit point of the sequence $\bar\phi^M$ (with respect to direct product (Tichonov)
topology on $(H^\Gamma)^n$. Now, $\bar x=(\bar\phi,\bar u)$ gives a solution we are looking for.

\begin{lemma}\label{lm_solution}
$\bar w((\bar f,\bar a),(\bar \phi,\bar u))=1 $.
\end{lemma}  

\begin{proof}
  We start with some preliminary considerations. As $\eta_N=\eta_{M,N}\circ\eta_M$ for $M\subseteq N$ we get that
  $\eta_M(\Phi_N)$ is a section for $\eta_{M,N}:G_M\to G_N$. As above, the restriction to $\eta_M(\Phi_N)$ and pullback
  by $\eta_{M.N}$ defines a map $H^{G_M}\to H^{G_N}$, $\psi\to\psi_N$. In other words $\psi_N(x_N)=\psi(x)$ for $x\in\eta_M(\Phi_N)$.
  This defines a map $H\wr G_M\to H\wr G_N$.
  
  By Lemma~\ref{lm_pr} we get $\bar w((\bar f,\bar a),(\bar \phi,\bar u))=(\bar \delta,1)$.
  Given $x\in\Gamma$ we need to show that $\bar \delta(x)=1$.
  Let $S=\bigcup\limits_{w\in\bar w}\Suf(w)$. Take $N\in\M$ such that $xS(\bar a,\bar u)\subseteq \Phi_N$. 
  By construction of $\bar \phi$ it follows that $\bar \phi|_{\Phi_N}=\bar \phi^M|_{\Phi_N}$ for some $M\in\O$, $M\subseteq N$.
  So, $\bar\phi_N=\bar\psi^M_N$ (the definition of $\bar \psi^M$ is in the construction of $\bar\phi$).  
  Let $\bar w((\bar f_N,\bar a_N),(\bar \psi^M_N,\bar u_N))=(\bar \delta',1)$. By the above consideration and
  Lemma~\ref{lm_local2} we obtain that $\bar \delta'(x_N)=\bar \delta(x)$. On the other hand, $x_MS(\bar a_M,\bar u_M)\subseteq \eta_M(\Phi_N)$
  and $\bar w((\bar f_M,\bar a_M),(\bar \psi^M,\bar u_M))=1$ by definition of $\psi^M$. Another application of Lemma~\ref{lm_local2} implies
  $\bar \delta'(x_N)=1$. 
\end{proof}  
\section{Concluding remarks}
The question ``if residually-finite-by-residually finite extensions are sofic'' remains unanswered. Although there is similar characterization of
sofic groups: A group $G$ is sofic if and only if every equation solvable in all permutation groups is solvable over $G$. The problem is that solvability in permutation
groups is not enough to prove, say, the existence of universal solutions.
\bibliography{sofic}

\def\cprime{$'$}
\begin{thebibliography}{10}
\expandafter\ifx\csname url\endcsname\relax
  \def\url#1{\texttt{#1}}\fi
\expandafter\ifx\csname urlprefix\endcsname\relax\def\urlprefix{URL }\fi
\expandafter\ifx\csname href\endcsname\relax
  \def\href#1#2{#2} \def\path#1{#1}\fi

\bibitem{Gr99}
M.~Gromov, \href{http://dx.doi.org/10.1007/PL00011162}{Endomorphisms of
  symbolic algebraic varieties}, J. Eur. Math. Soc. (JEMS) 1~(2) (1999)
  109--197.
\newblock \href {http://dx.doi.org/10.1007/PL00011162}
  {\path{doi:10.1007/PL00011162}}.
\newline\urlprefix\url{http://dx.doi.org/10.1007/PL00011162}

\bibitem{W00}
B.~Weiss, Sofic groups and dynamical systems, Sankhy\=a Ser. A 62~(3) (2000)
  350--359, ergodic theory and harmonic analysis (Mumbai, 1999).

\bibitem{Jaikin}
A.~Jaikin~Zapirain, personal communication, at Measured group theory, Thematic
  program, ESI, Wien, Austria 2016.01-2016.03.

\bibitem{mathover}
{Jon Bannon}, Is deligne's central extension sofic?, mathoverflow question,
  \url{https://mathoverflow.net/questions/48656/is-delignes-central-extension-sofic},
  [Online; accessed 10-October-2019] (2011).

\bibitem{EL2}
G.~Elek, E.~Szab\'{o}, \href{https://doi.org/10.1515/JGT.2006.011}{On sofic
  groups}, J. Group Theory 9~(2) (2006) 161--171.
\newblock \href {http://dx.doi.org/10.1515/JGT.2006.011}
  {\path{doi:10.1515/JGT.2006.011}}.
\newline\urlprefix\url{https://doi.org/10.1515/JGT.2006.011}

\bibitem{GLTC}
M.~{De Chiffre}, L.~{Glebsky}, A.~{Lubotzky}, A.~{Thom}, {Stability, cohomology
  vanishing, and non-approximable groups}, arXiv e-prints (2017)
  arXiv:1711.10238\href {http://arxiv.org/abs/1711.10238}
  {\path{arXiv:1711.10238}}.

\bibitem{Pestov1}
V.~G. Pestov, \href{http://dx.doi.org/10.2178/bsl/1231081461}{Hyperlinear and
  sofic groups: a brief guide}, Bull. Symbolic Logic 14~(4) (2008) 449--480.
\newblock \href {http://dx.doi.org/10.2178/bsl/1231081461}
  {\path{doi:10.2178/bsl/1231081461}}.
\newline\urlprefix\url{http://dx.doi.org/10.2178/bsl/1231081461}

\bibitem{GR}
L.~Glebsky, L.~M. Rivera,
  \href{http://dx.doi.org/10.1016/j.jalgebra.2008.08.008}{Sofic groups and
  profinite topology on free groups}, J. Algebra 320~(9) (2008) 3512--3518.
\newblock \href {http://dx.doi.org/10.1016/j.jalgebra.2008.08.008}
  {\path{doi:10.1016/j.jalgebra.2008.08.008}}.
\newline\urlprefix\url{http://dx.doi.org/10.1016/j.jalgebra.2008.08.008}

\bibitem{GL_AL}
L.~Glebsky,
  \href{https://doi.org/10.1016/j.jalgebra.2016.12.012}{Approximations of
  groups, characterizations of sofic groups, and equations over groups}, J.
  Algebra 477 (2017) 147--162.
\newblock \href {http://dx.doi.org/10.1016/j.jalgebra.2016.12.012}
  {\path{doi:10.1016/j.jalgebra.2016.12.012}}.
\newline\urlprefix\url{https://doi.org/10.1016/j.jalgebra.2016.12.012}

\bibitem{Ribes1}
L.~Ribes, P.~Zalesskii,
  \href{https://doi.org/10.1007/978-3-642-01642-4}{Profinite groups}, 2nd
  Edition, Vol.~40 of Ergebnisse der Mathematik und ihrer Grenzgebiete. 3.
  Folge. A Series of Modern Surveys in Mathematics [Results in Mathematics and
  Related Areas. 3rd Series. A Series of Modern Surveys in Mathematics],
  Springer-Verlag, Berlin, 2010.
\newblock \href {http://dx.doi.org/10.1007/978-3-642-01642-4}
  {\path{doi:10.1007/978-3-642-01642-4}}.
\newline\urlprefix\url{https://doi.org/10.1007/978-3-642-01642-4}

\end{thebibliography}

\end{document}